\documentclass[11pt]{article}

\usepackage[T1]{fontenc}
\usepackage[utf8]{inputenc}
\usepackage{lmodern}
\usepackage[margin=1.06in]{geometry}
\usepackage{microtype}
\usepackage{mathtools}
\usepackage{amssymb}
\usepackage{amsthm}
\usepackage{aliascnt}
\usepackage{enumitem}
\usepackage{needspace}
\usepackage[dvipsnames]{xcolor}
\usepackage[round,authoryear]{natbib}
\usepackage{hyperref}
\usepackage[nameinlink,noabbrev]{cleveref}

\hypersetup{
  pdfencoding=auto,
  colorlinks=true,
  linkcolor=MidnightBlue,
  citecolor=OliveGreen,
  urlcolor=BrickRed,
  pdftitle={Stability of Change-of-Numéraire Reweighting:
            An Exact Wasserstein Boundary},
  pdfauthor={Shaosai Huang},
  pdfsubject={Wasserstein stability of numéraire-weighted marginals},
  pdfkeywords={change of numeraire; annuity measure; Wasserstein
               distance; uniform integrability; weighted marginal;
               size-biasing; martingale optimal transport;
               self-normalized importance sampling; swaption}
}

\setlist[itemize]{leftmargin=2em,itemsep=0.3em,topsep=0.4em}
\setlist[enumerate]{leftmargin=2em,itemsep=0.3em,topsep=0.4em}
\allowdisplaybreaks

\newtheorem{theorem}{Theorem}[section]
\newaliascnt{lemma}{theorem}
\newtheorem{lemma}[lemma]{Lemma}
\aliascntresetthe{lemma}
\newaliascnt{proposition}{theorem}
\newtheorem{proposition}[proposition]{Proposition}
\aliascntresetthe{proposition}
\newaliascnt{corollary}{theorem}
\newtheorem{corollary}[corollary]{Corollary}
\aliascntresetthe{corollary}
\newaliascnt{assumption}{theorem}

\aliascntresetthe{assumption}
\theoremstyle{definition}
\newaliascnt{definition}{theorem}

\aliascntresetthe{definition}
\newaliascnt{example}{theorem}
\newtheorem{example}[example]{Example}
\aliascntresetthe{example}
\theoremstyle{remark}
\newaliascnt{remark}{theorem}
\newtheorem{remark}[remark]{Remark}
\aliascntresetthe{remark}
\crefname{theorem}{Theorem}{Theorems}
\Crefname{theorem}{Theorem}{Theorems}
\crefname{lemma}{Lemma}{Lemmas}
\Crefname{lemma}{Lemma}{Lemmas}
\crefname{proposition}{Proposition}{Propositions}
\Crefname{proposition}{Proposition}{Propositions}
\crefname{corollary}{Corollary}{Corollaries}
\Crefname{corollary}{Corollary}{Corollaries}
\crefname{assumption}{Assumption}{Assumptions}
\Crefname{assumption}{Assumption}{Assumptions}
\crefname{definition}{Definition}{Definitions}
\Crefname{definition}{Definition}{Definitions}
\crefname{example}{Example}{Examples}
\Crefname{example}{Example}{Examples}
\crefname{remark}{Remark}{Remarks}
\Crefname{remark}{Remark}{Remarks}
\crefname{section}{Section}{Sections}
\Crefname{section}{Section}{Sections}

\newcommand{\R}{\mathbb R}
\newcommand{\E}{\mathbb E}
\newcommand{\one}{\mathbf 1}
\newcommand{\dd}{\,\mathrm d}
\newcommand{\cP}{\mathcal P}
\newcommand{\cK}{\mathcal K}
\newcommand{\Ws}{W_s}
\newcommand{\Hplus}{[0,\infty)\times\R}
\newcommand{\Hopen}{(0,\infty)\times\R}

\title{Stability of Change-of-Numéraire Reweighting:\\
       An Exact Wasserstein Boundary%
       \thanks{Working paper.  Comments welcome.}}
\author{Shaosai Huang\thanks{Kspectra Research Inc., Toronto, ON,
M2N 0G3, Canada.  Email:
\texttt{arthur.foxie.huang@kspectra.ai}.}}
\date{2026-08-28}

\begin{document}

\renewcommand{\thefootnote}{\fnsymbol{footnote}}
\maketitle
\renewcommand{\thefootnote}{\arabic{footnote}}
\setcounter{footnote}{0}

\begin{abstract}
Reweighting a probability law by a positive numéraire and pushing
forward the payoff-to-numéraire ratio yields the change-of-numéraire
reweighting: the swap-rate law under the annuity measure, the
numéraire-inversion involution of martingale optimal transport, and
the population form of self-normalized importance sampling.  We characterize exactly when it
is Wasserstein-stable.  Along weakly convergent inputs with uniformly
integrable numéraires and a strictly positive limiting numéraire mean,
convergence of the reweighted laws is equivalent to a uniform
integrability condition, computed under the inputs, on one explicit
family in the payoff and numéraire.  At first order the numéraire
cancels, the sole obstruction being payoff mass where the numéraire
vanishes; absent such mass, prices and first moments pass to the limit
under a uniformly integrable payoff alone, with no assumption on the
reciprocal numéraire.  Two-atom examples with bounded inputs attain
the threshold exactly.  Applications characterize a standing moment
assumption in the transport literature and delimit which
annuity-measure statistics survive.
\end{abstract}

\medskip
\noindent\textbf{Keywords.}  Change of numéraire; annuity measure;
Wasserstein distance; uniform integrability; weighted marginal;
size-biasing; martingale optimal transport; self-normalized importance
sampling; swaption.

\medskip
\noindent\textbf{MSC 2020.}  Primary 60B10; secondary 28A33, 49Q22,
91G20, 91G30.

\medskip
\noindent\textbf{JEL classification.}  G13; G12; C02; C65.

\section{Introduction}
\label{sec:nr-intro}

Let $Q$ be a probability law carrying a pair $(A,B)$ with
$A>0$ and $\E_Q[A]\in(0,\infty)$.  The \emph{change-of-numéraire
reweighting} of $Q$ is the probability measure
\begin{equation}
 \Gamma(Q)
 :=Z_\#\bigl(D\dd Q\bigr),
 \qquad
 D:=\frac{A}{\E_Q[A]},
 \qquad
 Z:=\frac BA .
 \label{eq:nr-gamma}
\end{equation}
This object is ubiquitous.  In interest-rate markets, $A$ an annuity
and $B$ the difference of swap legs make $\Gamma(Q)$ the law of the
swap rate under the annuity measure, the measure in which swaptions
are quoted \citep{Neuberger1990,GemanElKarouiRochet1995,
Jamshidian1997}.  In martingale optimal transport, the case
$B\equiv1$, $A=X$ is the numéraire-inversion involution
$S(\mu)=\mathrm{law}\bigl(1/X\ \text{under}\ (X/\E[X])\dd\mu\bigr)$
of \citet{CampiLaachirMartini2017}, recently extended to weak
transport and continuous time by \citet{BeiglboeckPammerRiess2024}
and, with a general reweighting function, by
\citet{BackhoffLoeperObloj2024}; see also
\citet{HuesmannStebegg2018}, where the weighted inversion is singled
out among all monotonicity-preserving transformations of martingale
transport.  In statistics, \eqref{eq:nr-gamma} with $B\equiv A\cdot
Z$ is the population form of a self-normalized importance-sampling
target, and for $B=A^2$ it is size-biasing
\citep{ArratiaGoldsteinKochman2019}.

\subsection{Main results}
\label{subsec:nr-main}

We are not aware of a base-law characterization of Wasserstein
$\Ws$-continuity for normalized $A$-reweighting followed by the
quotient map $B/A$ that also allows the weak limit to charge
$\{A=0\}$.  This note gives such a characterization: an exact
two-way boundary carried by a single scalar family.  Under weak
convergence of the input laws with uniformly
integrable numéraires and a strictly positive limiting numéraire mean
(\cref{lem:nr-weak}), $\Ws$-convergence of the weighted marginals is
\emph{equivalent} to vanishing, uniformly in the sequence index, of the
base-law ratio tails of
\begin{equation}
 |B|^s A^{1-s}
 \label{eq:nr-boundary-family}
\end{equation}
on the event $\{|B_n|>RA_n\}$ (\cref{thm:nr-boundary}).  When the
limit does not charge $\{A=0\}$, this is equivalent to ordinary
uniform integrability of the family
\eqref{eq:nr-boundary-family} (\cref{cor:nr-sufficient}).  At
$s=1$ that family collapses to $|B_n|$: uniform integrability of the
payoff leg then gives $W_1$-convergence with \emph{no moment condition
on the inverse numéraire $1/A_n$}.  More generally, the only residual
$s=1$ boundary obstruction is nonzero payoff mass at $A=0$.  This is
an exact cancellation, not an estimate: at the level of call prices
and first moments, the numéraire divides out identically.

The boundary is attained.  For each prescribed $s>1$, a two-atom
example with all input variables in $[0,1]^2$ --- so every
nonnegative-order moment of the primitive coordinates is uniformly
bounded --- has weighted marginals converging in
$W_\sigma$ exactly for $1\leq\sigma<s$, failing at $s$, and having
divergent $\sigma$-moments for $\sigma>s$
(\cref{ex:nr-two-atom}).  The mechanism is a vanishing-numéraire
state: a scenario of probability $\varepsilon$ in which $A$ collapses
while $B$ stays of order one carries the entire $s$-th weighted
moment.  Nothing escapes to infinity on the input side; the escape is
created by the ratio-and-reweight map itself.  Two further examples
delimit the hypotheses: if the limit law charges the
vanishing-numéraire boundary $\{A=0\}$ with nonzero payoff mass, the
$s=1$ cancellation fails even for bounded inputs
(\cref{ex:nr-boundary-mass}); and without
uniform integrability of the numéraires, even weak stability of the
weighted marginals fails (\cref{ex:nr-numeraire-ui}).

The normalized reweighting itself is classical; the additional issue
here is its interaction with the quotient.  Before the quotient
pushforward, $Q\mapsto A\,Q/\E_Q[A]$ is a size-biased law and admits a
one-point Palm representation.  Kallenberg's continuity theorem gives
a three-way relation for the retained random measure, its mean mass,
and its normalized mass-biased law; the forward implication uses the
uniform integrability supplied by weak convergence together with
convergence of the mean masses
\citep[Section~10, Theorem~10.5]{Kallenberg1974}; see also
\citet[Chapter~6]{Kallenberg2017}.  This is not a three-way
equivalence for an arbitrary marked law of $(A,B)$: the one-point
random-measure encoding identifies all states with $A=0$ and therefore
cannot recover the mark $B$ there.

In sequential Monte Carlo the normalized-reweighting operator is the
Boltzmann--Gibbs or selection map, whose stability is classical in
total variation and in weighted total-variation norms
\citep{DelMoral2004,DelMoral2013,DelMoralHortonJasra2023} and for
which genuine Wasserstein estimates exist
\citep{Whiteley2021,Borghi2026}.  Three features separate those
selection/filter results from what is proved here.  They are
one-directional sufficient bounds rather than characterizations, and
the bounds that give a Lipschitz constant for the reweighting itself
obtain it from two-sided control of the weight --- in the
total-variation estimates the mean of the weight sits in a
denominator, so those bounds degenerate exactly where
\cref{cor:nr-sufficient} operates.  Results reaching beyond bounded
weights answer a different question: contraction in time for a fixed
model, rather than continuity along a varying sequence of models.
They reweight and retain the underlying variable; they do not push
forward by the ratio $B/A$.  Consequently they do not identify the
extra boundary term created when the quotient pushforward meets
$B$-mass on the zero-weight fiber $\{A=0\}$.  What follows is
accordingly an equivalence rather than a bound, for the ratio
marginal, allowing unbounded weights and a limit that may charge the
zero-weight fiber; \cref{subsec:nr-honest} places it against each
literature in turn.

Two consequences are immediate.  In the change-of-numéraire
literature, \citet{CampiLaachirMartini2017} and
\citet{BeiglboeckPammerRiess2024} prove algebraic and structural
theorems about $S$ and \emph{assume} what they need about moments of
the transformed marginal: the standing finite-second-moment
hypothesis on $S(\nu)$ in \citet{BeiglboeckPammerRiess2024} is
precisely the $s=2$, $B\equiv1$ instance of
\eqref{eq:nr-boundary-family}, and \citet{BackhoffLoeperObloj2024}
impose sufficient inverse-moment conditions ($\mu\in\cP_{-1,1}$).
\Cref{cor:nr-cn} characterizes exactly when such hypotheses survive a
limit.  On the sampling side, \citet{CosteGoldman2026} prove sharp
rates for the self-normalized importance-sampling \emph{empirical}
measure in $W_p$ under bounded densities; the present note is the
population-level companion of that question, in the unbounded-weight
regime where \eqref{eq:nr-boundary-family} binds.

\subsection{Why the boundary matters in the motivating application}
\label{subsec:nr-why}

For annuity-weighted swap-rate marginals the dichotomy is not a
technicality; it separates two computations practitioners treat very
differently.  Physically settled swaption prices and every
$W_1$-continuous functional of the annuity-measure marginal sit at
$s=1$, where the
theorem says: for weakly convergent primitive laws with uniformly
integrable annuities and legs and a boundary-free limit, prices move
continuously with no separate inverse-annuity moment hypothesis.
Related finite-model calculations, such as the discounting-switch
comparison of \citet{Piterbarg2020}, express prices under a common
annuity measure through an annuity-ratio density and approximate its
conditional projection onto the swap rate, in the tradition of the
annuity-mapping framework of
\citet[Chapter~16]{AndersenPiterbarg2010}.  Our result addresses a
different, asymptotic question: convergence of the primitive laws
implies convergence of physically settled prices under the stated
uniform-integrability hypotheses.  It neither prices a finite
discounting switch nor bounds a ratio-approximation error.

The raw quadratic statistic
$\int z^2\dd\Gamma=\E_Q[B^2/A]/\E_Q[A]$, which enters second-order
CMS-type approximations \citep{Hagan2003}, sits at $s=2$.
In its stylized annuity interpretation, \cref{ex:nr-two-atom}
exhibits bounded annuities and legs for which
the entire normalized call-price curve converges while this statistic
stays wrong by a fixed amount.  This is the same higher-moment integrability
obstruction highlighted, in a different setting, by the
moment-explosion analysis of \citet{AndersenPiterbarg2007}; it also
connects with the moment-extraction discussion of
\citet[fn.~16]{TrolleSchwartz2014}.  Cash-settled swaptions require a
different interpretation.  Provided $G(S)>0$ and
$0<\E[G(S)]<\infty$, their payoff $G(S)(S-K)^+$ remains
algebraically within the perspective structure after setting
$A=G(S)$ and $B=S G(S)$, so the $s=1$ cancellation applies to the
normalized $G$-weighted law.  But $G(S)$ is not a traded numéraire,
so this weighting is not supplied as a market pricing measure by the
change-of-numéraire theorem.  In the reverse no-arbitrage problem of
recovering a forward density from quoted cash-settled prices,
inverse-annuity integral conditions appear in \citet{Mercurio2008}.

\subsection{Mathematical decomposition and related work}
\label{subsec:nr-honest}

The map $Q\mapsto\Gamma(Q)$ first biases the input law by
$A/\E_Q[A]$ and then pushes the retained pair $(A,B)$ forward by the
quotient $B/A$.  For nonnegative Borel $f$,
\[
 \int f\,\dd\Gamma(Q)
 =\frac{\E_Q[A f(B/A)]}{\E_Q[A]}.
\]
Taking $f(z)=(z-K)^+$ or $f(z)=|z|^s$ gives the annuity-measure
pricing identity or
$\int|z|^s\dd\Gamma=\E_Q[|B|^sA^{1-s}]/\E_Q[A]$.
The call identity follows from change of numéraire
\citep{Neuberger1990,GemanElKarouiRochet1995,Jamshidian1997,
AndersenPiterbarg2010,Hagan2003}; analytically, the formula combines
a perspective integrand with the defining identity of a weighted law
\citep{Combettes2018,ArratiaGoldsteinKochman2019}.  The continuity
question concerns their composition, because $B/A$ can become large
along states where the reweighting factor $A$ is small while $B$ does
not vanish.

Under \ref{it:nr-s1}--\ref{it:nr-s2},
\cref{lem:nr-weak} first gives $\Gamma_n\Rightarrow\Gamma$, even when
the limiting base law charges $\{A=0\}$; the reweighting discards that
fiber.  Pulling Villani's weak-plus-moment characterization of
Wasserstein convergence through the moment identity then gives, for
every $s\geq1$,
\[
 \begin{aligned}
 &\Gamma\in\cP_s(\R),\quad
   \Gamma_n\in\cP_s(\R)\ \text{for every }n,\quad
   \Ws(\Gamma_n,\Gamma)\to0
 \\
 &\hspace{4em}\Longleftrightarrow\quad
   \lim_{R\to\infty}\sup_n
   \E\!\left[
   |B_n|^sA_n^{1-s}\one_{\{|B_n|>RA_n\}}
   \right]=0.
 \end{aligned}
\]
The weak-convergence step is related to the change-of-law argument in
Le Cam's third lemma \citep[Theorem~6.6]{vanderVaart1998}; the moment
step uses \citet[Theorem~6.9]{Villani2009}, with related formulations
in \citet{KraetschmerSchiedZaehle2014} and
\citet{FeinbergKasyanovLiang2020}.  If $Q(\{a=0\})=0$, the ratio-tail
condition is equivalent to ordinary uniform integrability of
$\{|B_n|^sA_n^{1-s}\}$.  At $s=1$, uniform integrability of
$\{|B_n|\}$ leaves exactly the boundary residue:
\[
 W_1(\Gamma_n,\Gamma)\to0
 \quad\Longleftrightarrow\quad
 \int_{\{a=0\}}|b|\,\dd Q=0.
\]
Thus mass at $(0,0)$ is harmless, whereas nonzero payoff mass on the
zero-numéraire fiber remains in primitive first moments but is absent
from the weighted limit.

Palm and size-bias results describe the first operation rather than
the quotient composition.  In the one-point specialization,
Kallenberg relates a retained random measure, its mean mass, and its
normalized mass-biased law
\citep[Section~10, Theorem~10.5]{Kallenberg1974}; see also
\citet[Chapter~6]{Kallenberg2017}.  Zero-mass states are permitted and
discarded, but the encoding identifies all $B$-marks on $\{A=0\}$ and
does not apply $B/A$.  For a family of nonnegative variables whose
means are bounded away from zero,
\citet[Theorem~8.1]{ArratiaGoldsteinKochman2019} equate uniform
integrability with tightness of the size-biased family, through the
identity $\E[X;X>L]=\E[X]\,\mathbb P(X^*>L)$ that the proofs below
also use; \cref{thm:nr-boundary} is that criterion one level up.
\Citet{Gnedin1998} likewise allow mass loss for size-biased
permutations, representing it by an added zero component.  These are
weak or tightness statements for retained objects; the result here is
a $\Ws$ criterion after the quotient pushforward.

Other neighboring theories contain parts of the same algebra.
Latz's condition (A5) places a prior-integrable envelope on
$\|\theta\|^pA(y\mid\theta)$, which becomes $|B|^sA^{1-s}$ for
$B=\theta A$ and $s=p$ \citep[Theorem~19]{Latz2020}.  That theory
assumes a positive likelihood; its survey discusses noiseless
degeneracies, including a subcase with vanishing normalization,
whereas the limiting normalization here remains positive
\citep[\S7.2]{Latz2023}.  Related posterior results give sufficient
stability bounds rather than this boundary equivalence
\citep{Sprungk2020,CvetkovicLie2025,DoleraMainini2023}.  When $A$ and
$B$ are densities, the mixed moment is a power integral adjacent to
contiguity and Hellinger-process theory, although Hellinger integrals
proper have orders in $(0,1)$ rather than $s\geq1$
\citep[Chapter~IV]{JacodShiryaev2003}.  Martingale-transport stability
instead perturbs unweighted marginals of a coupling problem
\citep{BackhoffPammer2022,Wiesel2023,
BeiglboeckJourdainMargheritiPammer2023,JourdainPammer2023};
the instability examples of \citet{BrueckerhoffJuillet2022} concern
dimension rather than reweighting.

The boundary residue also has a perspective-functional
interpretation.  The closed extension of
$a\phi(b/a)$, with $\phi(z)=|z|^s$, places a recession term on $a=0$
\citep{Spector2011,LieroMielkeSavare2018,SavareSodini2024}, whereas
the limiting quotient law $\Gamma$ assigns that fiber zero weight.
At $s=1$, the discrepancy between the perspective numerator and the
weighted-law numerator is exactly
$\int |b|\dd Q-\E_Q[A]\int|z|\dd\Gamma
=\int_{\{a=0\}}|b|\dd Q$.  Generalized Young-measure theory gives a
related local two-sided criterion through vanishing concentration,
with spatial tightness needed for a global conclusion
\citep[Theorem~2.9]{AlibertBouchitte1997}.  For $s>1$, extended-valued
lower-semicontinuity results accommodate the superlinear perspective,
but do not yield a two-sided Wasserstein criterion for the normalized
quotient law.  The weak-plus-homogeneous-moment template appears in
balanced transport \citep[Proposition~7.1.5]{AmbrosioGigliSavare2008}
and in unbalanced transport
\citep[Theorem~A.5]{SavareSodini2024}.  Here it yields the exact
base-law ratio-tail equivalence, its ordinary-uniform-integrability
form off the boundary, and the first-order boundary residue.  The
examples attain the exponent and separate boundary failure from
failure of numéraire uniform integrability; \cref{thm:nr-joint}
packages the criterion into joint continuity over compact model and
parameter classes.

\subsection{Structure}

\Cref{sec:nr-identity} fixes the setting and records the
cancellation identities.  \Cref{sec:nr-stability} proves the weak
stability lemma, the two-way boundary theorem, and its sharp
ordinary-uniform-integrability form.  \Cref{sec:nr-sharp} presents the three
examples.  \Cref{sec:nr-joint} gives the parameterized
joint-continuity theorem.  \Cref{sec:nr-applications} works out the
two applications.  An AI-use disclosure ends the paper.

\section{The reweighting map and the cancellation identities}
\label{sec:nr-identity}

Everything below rests on a single identity: weighting by $A$ and
dividing by $A$ cancel, so each moment of $\Gamma$ is a base-law
expectation of an explicit integrand.  We record it in the generality
the stability theory will need, and fix the notation used throughout.

Fix a probability law $Q$ on a measurable space carrying measurable
$A>0$ and $B\in\R$ with $A^0:=\E_Q[A]\in(0,\infty)$, and define
$D$, $Z$, $\Gamma=\Gamma(Q)$ by \eqref{eq:nr-gamma}.  Since
$\E_Q[D]=1$ and $D>0$, $\Gamma$ is a probability measure on $\R$.

\begin{proposition}[Perspective cancellation]
\label{prop:nr-cancellation}
The following hold as identities in $[0,\infty]$.  For every
$K\in\R$,
\begin{equation}
 \int(z-K)^+\,\Gamma(\dd z)
 =\frac{\E_Q\bigl[(B-KA)^+\bigr]}{A^0},
 \qquad
 \int|z|\,\Gamma(\dd z)
 =\frac{\E_Q|B|}{A^0},
 \label{eq:nr-call-identity}
\end{equation}
and for every $s>1$,
\begin{equation}
 \int|z|^s\,\Gamma(\dd z)
 =\frac{\E_Q\bigl[|B|^s A^{1-s}\bigr]}{A^0}.
 \label{eq:nr-moment-identity}
\end{equation}
More generally, for every measurable $f\geq0$,
\begin{equation}
 \int f\,\dd\Gamma=\frac{\E_Q[A\,f(B/A)]}{A^0}.
 \label{eq:nr-general-identity}
\end{equation}
In particular calls and the first moment are finite whenever
$\E_Q|B|<\infty$, with no condition on $1/A$; the $s$-moment is
finite exactly when $\E_Q[|B|^sA^{1-s}]<\infty$.
\end{proposition}

\begin{proof}
Equation \eqref{eq:nr-general-identity} is the definition of
$\Gamma$.  Positive homogeneity gives
$A\,(B/A-K)^+=(B-KA)^+$ and $A\,|B/A|^s=|B|^sA^{1-s}$, proving
\eqref{eq:nr-call-identity} and \eqref{eq:nr-moment-identity}.
\end{proof}

\begin{remark}[Lineage]
\label{rem:nr-lineage}
In the annuity instance, \eqref{eq:nr-call-identity} is the
classical swaption pricing identity
\citep{Neuberger1990,GemanElKarouiRochet1995,Jamshidian1997}.
Abstractly, the map $(A,B)\mapsto\E_Q[A\,f(B/A)]$ is a
perspective-function integral \citep{Combettes2018}, and
\eqref{eq:nr-general-identity} is the definitional identity of a
weighted law \citep{ArratiaGoldsteinKochman2019}.  We claim no
novelty for \cref{prop:nr-cancellation}; the note pivots on reading
it as a statement about \emph{which base-law functional controls
which $\Gamma$-moment}: at $s=1$ the numéraire cancels exactly, and
for $s>1$ the exact cost is the single scalar family
\eqref{eq:nr-boundary-family}.
\end{remark}

\section{The two-way stability theorem}
\label{sec:nr-stability}

\Cref{prop:nr-cancellation} converts each $\Gamma$-moment into a
base-law expectation.  This section turns that conversion into a
stability theory: we fix a mode of convergence for the input laws,
show that the weighted marginals then converge weakly at no further
cost, and identify which additional condition upgrades weak
convergence to $\Ws$.  The identification is an equivalence, so it
also says what cannot be weakened.

\subsection{Standing setting}

For $n\in\mathbb N$, let $Q_n$ be the law of a pair $(A_n,B_n)$
with values in $\Hopen$.  Let $Q$ be a probability law on $\Hplus$,
write $(A,B)$ for its coordinate map, and assume throughout this
section:
\begin{enumerate}[label=\textup{(S\arabic*)}]
\item\label{it:nr-s1} $Q_n\to Q$ weakly as laws on the closed
      half-plane $\Hplus$ (the limit is allowed to charge the
      boundary $\{0\}\times\R$ unless stated otherwise);
\item\label{it:nr-s2} the numéraires $\{A_n\}_{n<\infty}$ are
      uniformly integrable, and $m:=\int a\,\dd Q>0$.
\end{enumerate}
Under \ref{it:nr-s1}--\ref{it:nr-s2}, $\E[A_n]\to m$ (uniform
integrability upgrades weak convergence of the nonnegative marginals
to convergence of means).  Define $\Gamma_n:=\Gamma(Q_n)$ by
\eqref{eq:nr-gamma} for $n<\infty$, and define the limit object by
zero-weighting the boundary:
\begin{equation}
 \int f\,\dd\Gamma
 :=\frac1m\int_{\{a>0\}} a\,f(b/a)\,Q(\dd a,\dd b),
 \qquad f\in C_b(\R).
 \label{eq:nr-gamma-limit}
\end{equation}
Since $\int_{\{a>0\}}a\,\dd Q=\int a\,\dd Q=m$, $\Gamma$ is a
probability measure; when $Q(\{a=0\})=0$ it is $\Gamma(Q)$.
Let $\mu(\dd a,\dd b):=(a/m)\one_{\{a>0\}}Q(\dd a,\dd b)$ and define
$T(a,b):=b/a$ for $a>0$ and $T(0,b):=0$.  Then $T$ is Borel,
$\mu(\{a=0\})=0$, and equality in \eqref{eq:nr-gamma-limit} for
$f\in C_b(\R)$ identifies $\Gamma=T_\#\mu$.  Consequently,
\eqref{eq:nr-gamma-limit} holds for every nonnegative Borel $f$: first
for bounded Borel $f$ by equality of the measures, and then in general
by monotone convergence.

\subsection{Weak stability}

Weak convergence of the weighted marginals costs nothing beyond
\ref{it:nr-s1}--\ref{it:nr-s2}.  The argument is a truncation, carried
out in the boundary-permitting form \eqref{eq:nr-gamma-limit}: mass
that the limit places on $\{a=0\}$ simply receives no weight.  That
convention is not a technical convenience, and
\cref{ex:nr-boundary-mass} shows what it costs when the payoff does
not vanish there.

\begin{lemma}[Weak stability under numéraire uniform integrability]
\label{lem:nr-weak}
Under \ref{it:nr-s1}--\ref{it:nr-s2}, $\Gamma_n\to\Gamma$ weakly.
\end{lemma}

\begin{proof}
For $f\in C_b(\R)$ define $g_f:\Hplus\to\R$ by
$g_f(a,b):=a\,f(b/a)$ for $a>0$ and $g_f(0,b):=0$.  Then $g_f$ is
continuous: at interior points this is clear, and if
$(a_k,b_k)\to(0,b)$ then $|g_f(a_k,b_k)|\leq a_k\|f\|_\infty\to0$.
Moreover $|g_f|\leq\|f\|_\infty\,a$.  Fix $\varepsilon>0$ and choose
$M$ with $\sup_n\E[A_n\one\{A_n>M\}]\leq\varepsilon$, possible by
\ref{it:nr-s2}.  Let $\chi_M:[0,\infty)\to[0,1]$ be continuous with
$\chi_M=1$ on $[0,M]$ and $\chi_M=0$ on $[2M,\infty)$.  Then
$g_f\cdot\chi_M(a)\in C_b(\Hplus)$, so
$\E[g_f\chi_M(A_n,B_n)]\to\int g_f\chi_M\,\dd Q$ by
\ref{it:nr-s1}, while
$|\E[g_f(A_n,B_n)]-\E[g_f\chi_M(A_n,B_n)]|
 \leq\|f\|_\infty\,\E[A_n\one\{A_n>M\}]\leq\|f\|_\infty\varepsilon$
uniformly in $n$.  The same tail bound holds for the limit law by the
lower-semicontinuous Portmanteau theorem applied to
$a\one\{a>M\}$.  Hence
$\E[g_f(A_n,B_n)]\to\int g_f\,\dd Q$.  Taking $f\equiv1$ gives
$\E[A_n]\to m$, and therefore
$\int f\,\dd\Gamma_n=\E[g_f(A_n,B_n)]/\E[A_n]\to
 \bigl(\int g_f\,\dd Q\bigr)/m=\int f\,\dd\Gamma$.
\end{proof}

\begin{remark}
\label{rem:nr-lecam}
\Cref{lem:nr-weak} is a change-of-numéraire dress of Le Cam's third
lemma \citep[Theorem~6.6]{vanderVaart1998}: uniform integrability of
the normalized weights $D_n$ is the classical contiguity condition,
and \cref{ex:nr-numeraire-ui} below is the standard failure when it
is dropped.  We include the proof because the boundary-permitting
formulation \eqref{eq:nr-gamma-limit} --- the limit may sit mass on
$\{a=0\}$, which then receives zero weight --- is used essentially
in \cref{ex:nr-boundary-mass}.
\end{remark}

\Needspace{20\baselineskip}
\subsection{The boundary}

Weak convergence being free, the entire question is the $s$-th
moment, and \cref{prop:nr-cancellation} moves that question to the
base laws, where it concerns the single family
\eqref{eq:nr-boundary-family}.  The following theorem records the
answer as an equivalence: the displayed ratio tails are not one
sufficient condition among several, but exactly the price of
$\Ws$-convergence.

\begin{theorem}[Exact two-way $\Ws$ boundary]
\label{thm:nr-boundary}
Assume \ref{it:nr-s1}--\ref{it:nr-s2} and fix $s\geq1$.  Put
\[
 W_n:=|B_n|^sA_n^{1-s},
 \qquad
 I_s:=\int_{\{a>0\}}|b|^sa^{1-s}\,Q(\dd a,\dd b).
\]
The following are equivalent:
\begin{enumerate}[label=\textup{(\roman*)}]
\item\label{it:nr-i} $\Gamma\in\cP_s(\R)$,
$\Gamma_n\in\cP_s(\R)$ for every $n$, and
$\Ws(\Gamma_n,\Gamma)\to0$;
\item\label{it:nr-ii} $W_n$ is integrable for every $n$,
$I_s<\infty$, and $\E[W_n]\to I_s$;
\item\label{it:nr-iii} the base-side ratio tails vanish uniformly:
\begin{equation}
 \lim_{R\to\infty}\ \sup_{n\in\mathbb N}
 \E\bigl[W_n\one\{|B_n|>RA_n\}\bigr]=0.
 \label{eq:nr-condition-iii}
\end{equation}
\end{enumerate}
\end{theorem}

\begin{proof}
Write $m_n:=\E[A_n]$.  By \ref{it:nr-s1}--\ref{it:nr-s2},
$m_n\to m\in(0,\infty)$; hence
$0<\inf_n m_n\leq\sup_n m_n<\infty$.  By
\cref{lem:nr-weak}, $\Gamma_n\to\Gamma$ weakly, while
\cref{prop:nr-cancellation,eq:nr-gamma-limit} give, in
$[0,\infty]$,
\[
 \int|z|^s\,\dd\Gamma_n
 =\frac{\E[W_n]}{m_n},
 \qquad
 \int|z|^s\,\dd\Gamma=\frac{I_s}{m},
\]
and
\[
 \int_{\{|z|>R\}}|z|^s\,\dd\Gamma_n
 =\frac{\E[W_n\one\{|B_n|>RA_n\}]}{m_n}.
\]
The standard moment characterization of Wasserstein convergence
gives \ref{it:nr-i}$\Leftrightarrow$\ref{it:nr-ii}, and its
uniform-tail form gives
\ref{it:nr-i}$\Leftrightarrow$\ref{it:nr-iii}
\citep[Theorem~6.9]{Villani2009}.  For completeness,
\ref{it:nr-iii} also supplies the domains in \ref{it:nr-i}: for
some $R_0$ its supremum is finite, and
$W_n\one\{|B_n|\leq R_0A_n\}\leq R_0^sA_n$, so every $W_n$ is
integrable.  The corresponding uniform $s$-moment tails of
$\Gamma_n$, together with weak convergence and lower-semicontinuous
portmanteau, imply $\Gamma\in\cP_s(\R)$.
\end{proof}

The point of \cref{thm:nr-boundary} is that the exact stability
condition is a statement about the base laws: the single family
\eqref{eq:nr-boundary-family}, truncated on the ratio event
$\{|B_n|>RA_n\}$, decides $\Ws$-convergence.  The following gives
both its ordinary-uniform-integrability form away from the boundary
and the sharp residual obstruction at $s=1$.

\begin{corollary}[Sharp cancellation and ordinary uniform integrability]
\label{cor:nr-sufficient}
Assume \ref{it:nr-s1}--\ref{it:nr-s2}.
\begin{enumerate}[label=\textup{(\alph*)}]
\item\label{it:nr-a} Suppose $\{|B_n|\}$ is uniformly integrable.
Then $\Gamma\in\cP_1(\R)$ and $\Gamma_n\in\cP_1(\R)$ for every
$n$, and
\begin{equation}
 W_1(\Gamma_n,\Gamma)\to0
 \quad\Longleftrightarrow\quad
 \int_{\{a=0\}}|b|\,Q(\dd a,\dd b)=0.
 \label{eq:nr-s1-boundary}
\end{equation}
Thus boundary mass at $(0,0)$ is harmless; in particular the
conclusion holds when $Q(\{a=0\})=0$, with no hypothesis on
$\{1/A_n\}$.
\item\label{it:nr-b} Suppose $Q(\{a=0\})=0$ and fix $s\geq1$.
With $W_n:=|B_n|^sA_n^{1-s}$, the following are equivalent:
\begin{enumerate}[label=\textup{(\roman*)},leftmargin=2.5em]
\item $\Gamma\in\cP_s(\R)$, $\Gamma_n\in\cP_s(\R)$ for every $n$,
and
$\Ws(\Gamma_n,\Gamma)\to0$;
\item $\{W_n:n\in\mathbb N\}$ is uniformly integrable.
\end{enumerate}
\end{enumerate}
\end{corollary}

\begin{proof}
For \ref{it:nr-a}, uniform integrability and weak convergence give
\[
 \E|B_n|\longrightarrow\int_{\Hplus}|b|\,\dd Q.
\]
Since $\E[A_n]\to m$, \cref{prop:nr-cancellation} yields
\[
 \int|z|\,\dd\Gamma_n\longrightarrow
 \frac1m\int_{\Hplus}|b|\,\dd Q,
 \qquad
 \int|z|\,\dd\Gamma=
 \frac1m\int_{\{a>0\}}|b|\,\dd Q.
\]
Together with $\Gamma_n\to\Gamma$ weakly, the first-moment
characterization of $W_1$ proves \eqref{eq:nr-s1-boundary}.

For \ref{it:nr-b}, put $Z_n:=B_n/A_n$.  Boundary-freeness and weak
convergence imply
\begin{equation}
 \lim_{\delta\downarrow0}\sup_n Q_n(A_n\leq\delta)=0.
 \label{eq:nr-uniform-away-zero}
\end{equation}
Indeed, portmanteau controls all sufficiently large $n$, and each of
the finitely many remaining laws assigns no mass to $\{A_n=0\}$.
If $\{W_n\}$ is uniformly integrable, then for $M,R>0$,
\[
 \E[W_n\one\{|Z_n|>R\}]
 \leq \E[W_n\one\{W_n>M\}]
      +M Q_n(A_n\leq MR^{-s}).
\]
First choose $M$ and then $R$, using
\eqref{eq:nr-uniform-away-zero}, to obtain
\eqref{eq:nr-condition-iii}; apply \cref{thm:nr-boundary}.
Conversely, if $\Ws(\Gamma_n,\Gamma)\to0$, then
\cref{thm:nr-boundary} gives \eqref{eq:nr-condition-iii}, and
\[
 \E[W_n\one\{W_n>M\}]
 \leq \E[W_n\one\{|Z_n|>R\}]
      +R^s\E[A_n\one\{A_n>M/R^s\}].
\]
Choose $R$ by \eqref{eq:nr-condition-iii} and then $M$ by
\ref{it:nr-s2}.  Hence $\{W_n\}$ is uniformly integrable.
\end{proof}

\begin{remark}[Ratio tails versus ordinary uniform integrability]
\label{rem:nr-exact-vs-clean}
Under \ref{it:nr-s2}, the exact ratio-tail condition
\eqref{eq:nr-condition-iii} always implies ordinary uniform
integrability of $W_n=|B_n|^sA_n^{1-s}$, by the second inequality in
the proof above.  The converse holds when the limiting law does not
charge $\{a=0\}$, by
\cref{cor:nr-sufficient}\ref{it:nr-b}, but can fail at the boundary:
\cref{ex:nr-boundary-mass} has $W_n=|B_n|\leq1$ while
\eqref{eq:nr-condition-iii} fails.  At $s=1$,
\cref{cor:nr-sufficient}\ref{it:nr-a} identifies the residual
obstruction exactly as $\int_{\{a=0\}}|b|\,\dd Q$.
\end{remark}

\section{Sharpness}
\label{sec:nr-sharp}

Neither the hypotheses of \cref{thm:nr-boundary,cor:nr-sufficient} nor
the exponent in \eqref{eq:nr-boundary-family} is decorative.  Three
two-atom sequences, all computed in closed form, settle this.  The
first attains the threshold exactly at any prescribed $s>1$ while
confining every primitive coordinate to $[0,1]$; the second shows that
the $s=1$ cancellation is destroyed by payoff mass on the
vanishing-numéraire boundary, and only by that; the third shows that
numéraire uniform integrability cannot be dropped even for weak
convergence.

\begin{example}[The boundary is attained: bounded inputs, exact
threshold]
\label{ex:nr-two-atom}
Fix $s>1$.  Let $\varepsilon_n\downarrow0$,
$a_n:=\varepsilon_n^{1/(s-1)}$, and let $Q_n$ put mass
$1-\varepsilon_n$ at $(A,B)=(1,0)$ and mass $\varepsilon_n$ at
$(A,B)=(a_n,1)$.  All inputs take values in $[0,1]^2$, so every
nonnegative-order moment of the primitive coordinates is uniformly
bounded and
\ref{it:nr-s1}--\ref{it:nr-s2} hold with limit
$Q=\delta_{(1,0)}$ and $Q(\{a=0\})=0$; also
$\E[A_n]=1-\varepsilon_n+\varepsilon_na_n\to1$.  The weighted
marginal $\Gamma_n$ has an atom at $0$ of mass
$(1-\varepsilon_n)/\E[A_n]$ and an atom at $1/a_n$ of mass
$\varepsilon_na_n/\E[A_n]\to0$; hence $\Gamma_n\to\delta_0=\Gamma$
weakly.  For every $\sigma\geq1$,
\[
 \int|z|^\sigma\,\dd\Gamma_n
 =\frac{\varepsilon_n\,a_n^{1-\sigma}}{\E[A_n]}
 =\frac{\varepsilon_n^{(s-\sigma)/(s-1)}}{\E[A_n]}
 \longrightarrow
 \begin{cases}
  0, & 1\leq\sigma<s,\\[0.2em]
  1, & \sigma=s,\\[0.2em]
  +\infty, & \sigma>s.
 \end{cases}
\]
Thus $\Gamma_n\to\delta_0$ in $W_\sigma$ for every $\sigma<s$
(in particular the $s=1$ cancellation operates:
the first moment is $\varepsilon_n/\E[A_n]\to0$, matching
\cref{prop:nr-cancellation} exactly), while the $s$-th moment
converges to $1\neq0$ and $\Ws$-convergence fails.  Consistently,
the family $\{|B_n|^\sigma A_n^{1-\sigma}\}$ is uniformly integrable
for every $\sigma<s$: on the rare atom its tail expectation is at
most $\varepsilon_n a_n^{1-\sigma}
=\varepsilon_n^{(s-\sigma)/(s-1)}\to0$, and the finitely many
remaining indices are harmless.  It
first fails to be uniformly integrable at $\sigma=s$, where
the single rare atom carries $\E[W_n^{(s)}]=1$; here
$W_n^{(\sigma)}:=|B_n|^\sigma A_n^{1-\sigma}$.
\end{example}

\begin{remark}[Reading of \cref{ex:nr-two-atom}]
\label{rem:nr-reading}
Viewed on the $\Gamma$-side this is the classical escape-to-infinity
moment example; the content is on the $Q$-side, where \emph{nothing}
escapes: both coordinates live in $[0,1]$, all nonnegative-order
moments of the primitive coordinates are bounded, and the joint laws
converge.  The escape is manufactured
entirely by the ratio-and-reweight map in a vanishing-numéraire
state of probability $\varepsilon_n$ in which the payoff leg stays
of order one.  In a stylized annuity interpretation, this is a rare
scenario in which the annuity collapses while the leg value does not.
The entire normalized call-price curve converges (an $s=1$ object,
protected by the cancellation).  With $s=2$ and
$a_n=\varepsilon_n$, however, the raw quadratic statistic
$\int z^2\dd\Gamma_n$ stays asymptotically one unit away from its
limiting value.
This statistic enters second-order CMS-type approximations
\citep{Hagan2003}; any such approximation with nonzero exposure to
it can therefore fail to converge.  A particular CMS correction must
still be assessed by its payoff growth.  Bounded primitive inputs do
not imply stable weighted moments: normalized call-price curves can
converge while a quadratic statistic stays wrong by a fixed amount.
\end{remark}

\begin{example}[Nonzero payoff mass on the vanishing-numéraire
boundary defeats $W_1$-stability]
\label{ex:nr-boundary-mass}
Fix $\varepsilon\in(0,1)$, let $\alpha_n\downarrow0$, and let $Q_n$
put mass $1-\varepsilon$ at $(1,0)$ and mass $\varepsilon$ at
$(\alpha_n,1)$.  Then \ref{it:nr-s1}--\ref{it:nr-s2} hold with
limit $Q=(1-\varepsilon)\delta_{(1,0)}+\varepsilon\delta_{(0,1)}$,
which charges $\{a=0\}$; $m=1-\varepsilon>0$;
$\{|B_n|\}$ is bounded by $1$, hence uniformly integrable.
\Cref{lem:nr-weak} applies and indeed
$\Gamma_n\to\Gamma=\delta_0$ weakly (the escaping atom at
$1/\alpha_n$ has $\Gamma_n$-mass
$\varepsilon\alpha_n/\E[A_n]\to0$).  But
\[
 \int|z|\,\dd\Gamma_n
 =\frac{\varepsilon}{\E[A_n]}
 \longrightarrow\frac{\varepsilon}{1-\varepsilon}
 \neq0=\int|z|\,\dd\Gamma,
\]
so $W_1$-convergence fails although the payoff legs are bounded by
one.  Consistently, \eqref{eq:nr-condition-iii} fails:
$\E[|B_n|\one\{|B_n|>RA_n\}]=\varepsilon$ once $R\alpha_n<1$.  The
boundary-payoff condition \eqref{eq:nr-s1-boundary} is therefore
sharp: the cancellation protects against vanishing-numéraire states
of vanishing probability, and boundary mass at $(0,0)$ is harmless,
but nonzero payoff mass on $\{a=0\}$ destroys $W_1$-convergence.
\end{example}

\begin{example}[Numéraire uniform integrability is needed for weak
stability]
\label{ex:nr-numeraire-ui}
Let $Q_n$ put mass $1-\varepsilon_n$ at $(1,0)$ and mass
$\varepsilon_n$ at $(1/\varepsilon_n,\,1/\varepsilon_n)$ with
$\varepsilon_n\downarrow0$.  Then $Q_n\to\delta_{(1,0)}$ weakly, but
$\{A_n\}$ is not uniformly integrable
($\E[A_n]=1-\varepsilon_n+1\to2\neq1$).  The weighted marginals are
\[
 \Gamma_n
 =\frac{1-\varepsilon_n}{\E[A_n]}\,\delta_0
 +\frac{1}{\E[A_n]}\,\delta_1
 \longrightarrow
 \tfrac12\delta_0+\tfrac12\delta_1
 \neq\delta_0=\Gamma:
\]
even \emph{weak} stability fails.  This is the classical failure of
contiguity when the likelihood ratios are not uniformly integrable
\citep[Chapter~6]{vanderVaart1998}, in numéraire dress: half of the
limiting weighted mass sits on an event the limiting base law does
not see.
\end{example}

\section{Joint continuity for parameterized contracts}
\label{sec:nr-joint}

In applications the pair $(A,B)$ carries a contract parameter ---
an expiry, a tenor, a settlement convention --- and the model law
and the parameter vary together.  The following packages
\cref{thm:nr-boundary,cor:nr-sufficient} in that setting.

\begin{theorem}[Joint continuity, compact image, affine slices]
\label{thm:nr-joint}
Let $\Omega$ be a Polish space, $\cK$ a compact set of Borel
probability laws on $\Omega$ (weak topology), and $\mathsf A$ a
compact metric space of contract parameters.  Let
\[
 (a,\omega)\longmapsto\bigl(A_a(\omega),B_a(\omega)\bigr)
 \in\Hopen
\]
be jointly continuous, and fix $s\geq1$.  Assume the uniform
weighted-tail condition
\begin{equation}
 \lim_{R\to\infty}\
 \sup_{(Q,a)\in\cK\times\mathsf A}
 \E_Q\bigl[\,|B_a|^sA_a^{1-s}\,
 \one\{|B_a|>R\,A_a\}\bigr]=0,
 \label{eq:nr-joint-tail}
\end{equation}
together with uniform integrability of
$\{A_a:\ (Q,a)\in\cK\times\mathsf A\}$ and
$\inf_{(Q,a)}\E_Q[A_a]>0$.  Then:
\begin{enumerate}[label=\textup{(\roman*)}]
\item the map
$(Q,a)\mapsto\Gamma_a(Q):=\Gamma$ of \eqref{eq:nr-gamma} applied to
the law of $(A_a,B_a)$ under $Q$ is continuous from
$\cK\times\mathsf A$ into $(\cP_s(\R),\Ws)$;
\item its image is $\Ws$-compact;
\item on every calibration slice
$\cK_c:=\{Q\in\cK:\E_Q[A_a]=c\}$ with $c>0$ and $a$ fixed, the map
$Q\mapsto\Gamma_a(Q)$ is affine whenever the relevant mixture
remains in $\cK_c$; on all of $\cK$ it is projective (a ratio of two
affine maps).  For $Q_1,Q_2$ and their mixture in $\cK$, mixtures
reweight by relative numéraire value: with
$w_1:=\lambda\,\E_{Q_1}[A_a]$ and
$w_2:=(1-\lambda)\,\E_{Q_2}[A_a]$,
\[
 \Gamma_a\bigl(\lambda Q_1+(1-\lambda)Q_2\bigr)
 =\frac{w_1\,\Gamma_a(Q_1)+w_2\,\Gamma_a(Q_2)}{w_1+w_2}.
\]
\end{enumerate}
\end{theorem}

\begin{proof}
(i)  Let $(Q_n,a_n)\to(Q,a)$ in $\cK\times\mathsf A$.  The product
laws $\delta_{a_n}\otimes Q_n$ converge weakly to
$\delta_a\otimes Q$ on $\mathsf A\times\Omega$, so by the continuous
mapping theorem applied to the jointly continuous map
$(a,\omega)\mapsto(A_a(\omega),B_a(\omega))$, the laws of
$(A_{a_n},B_{a_n})$ under $Q_n$ converge weakly to the law of
$(A_a,B_a)$ under $Q$.  The latter charges $\Hopen$, so its
$a$-marginal puts no mass at $0$.  The assumed uniform integrability
of the numéraires gives \ref{it:nr-s2}, and
\eqref{eq:nr-joint-tail} gives \eqref{eq:nr-condition-iii} along the
sequence.  \Cref{thm:nr-boundary} yields
$\Ws(\Gamma_{a_n}(Q_n),\Gamma_a(Q))\to0$.

(ii)  A continuous image of the compact set
$\cK\times\mathsf A$ is compact.

(iii)  Both
$Q\mapsto\E_Q[A_a\,f(Z_a)]$ and $Q\mapsto\E_Q[A_a]$ are affine in
$Q$; the displayed mixture formula follows by dividing, and on a
slice of constant $\E_Q[A_a]$ the ratio itself is affine.
\end{proof}

\begin{remark}[Why calibration slices are the natural domain]
\label{rem:nr-slice}
In the calibration instance the numéraire's time-zero value is
itself a quoted input: every candidate law prices the annuity
identically, so the relevant $\cK$ \emph{is} a slice $\cK_c$ and the
map is genuinely affine there.  Off the slice, affinity in $Q$ fails
in general --- the mixture formula reweights by relative numéraire
value --- which matters for barycentric constructions: averaging
models does not average their annuity-measure marginals unless the
annuity values agree.
\end{remark}

\section{Two applications}
\label{sec:nr-applications}

The two settings named in \cref{sec:nr-intro} are now instances.  In
the first, the boundary characterizes a moment hypothesis that the
change-of-numéraire literature currently assumes; in the second, it
separates the statistics of an annuity-measure marginal that survive a
model limit from those that do not.

\subsection{The change-of-numéraire involution of martingale
optimal transport}
\label{subsec:nr-cn}

For $\mu$ on $(0,\infty)$ with mean
$m(\mu)\in(0,\infty)$, the change-of-numéraire transform
\[
 S(\mu):=\text{law of }1/X\ \text{under}\ \tfrac{X}{m(\mu)}\dd\mu
\]
is the involution of \citet{CampiLaachirMartini2017} (there under a
unit-mean normalization), extended to weak martingale transport by
\citet{BeiglboeckPammerRiess2024} and, with general reweighting
functions, by \citet{BackhoffLoeperObloj2024}.  This is
\eqref{eq:nr-gamma} with $(A,B)=(X,1)$, and
\cref{prop:nr-cancellation} gives
\begin{equation}
 \int y^s\,\dd S(\mu)
 =\frac{\E_\mu[X^{1-s}]}{m(\mu)},
 \qquad s\geq 1 .
 \label{eq:nr-cn-moment}
\end{equation}

\begin{corollary}[Exact $\Ws$-stability of the involution]
\label{cor:nr-cn}
Let $\mu_n\to\mu$ weakly as probability laws on $(0,\infty)$,
assume $\{X_n\}$ is uniformly integrable and $m(\mu)>0$, and fix
$s\geq1$.  Then $S(\mu_n)\to S(\mu)$ weakly, and the following are
equivalent:
\begin{enumerate}[label=\textup{(\roman*)}]
\item $S(\mu)\in\cP_s(\R)$, $S(\mu_n)\in\cP_s(\R)$ for every $n$,
and
$\Ws(S(\mu_n),S(\mu))\to0$;
\item every expectation below is finite and
$\E_{\mu_n}[X_n^{1-s}]\to\E_\mu[X^{1-s}]$;
\item $\{X_n^{1-s}:n\in\mathbb N\}$ is uniformly integrable;
\item
\[
 \lim_{R\to\infty}\sup_n
 \E_{\mu_n}[X_n^{1-s}\one\{X_n<1/R\}]=0.
\]
\end{enumerate}
At $s=1$ these conditions are automatic, so $W_1$-stability needs
no assumption beyond the standing ones.
\end{corollary}

\begin{proof}
Apply \cref{lem:nr-weak,thm:nr-boundary,cor:nr-sufficient}%
\ref{it:nr-b} with $(A_n,B_n)=(X_n,1)$; then
$W_n=X_n^{1-s}$ and $\{|B_n|>RA_n\}=\{X_n<1/R\}$.
\end{proof}

Two readings.  First, the standing assumption
``$S(\nu)$ has finite second moment'' in
\citet{BeiglboeckPammerRiess2024} is, by \eqref{eq:nr-cn-moment},
exactly $\E_\nu[X^{-1}]<\infty$; \cref{cor:nr-cn} upgrades this
static observation to the exact stability statement: along weakly
convergent marginals with converging means, the transformed
marginals move continuously in $W_2$ precisely when the inverse
moments $\E[X_n^{-1}]$ converge, equivalently when
$\{X_n^{-1}\}$ is uniformly integrable.  The characterization is
attained already by the two-atom laws that put mass
$1-\varepsilon_n$ at $1$ and mass $\varepsilon_n$ at
$a_n=\varepsilon_n$; their transforms have the same exact $W_2$
threshold as \cref{ex:nr-two-atom}.  Second, the inverse-moment classes
$\cP_{-1,1}$ imposed in \citet{BackhoffLoeperObloj2024} are
sufficient-condition territory for $s=2$; the corollary locates the
exact boundary inside them.

\subsection{Annuity-measure swap-rate marginals}
\label{subsec:nr-annuity}

Let $a$ range over a compact family of swap contracts, $A_a>0$ the
annuity (a finite sum of accrual-weighted zero-coupon bonds --- hence,
for a contract class with uniformly bounded accrual sums, a common
bond bound gives a common bound on the annuities and therefore their
uniform integrability), and $B_a$ the difference of the
legs.  Then $\Gamma_a(Q)$ is the annuity-measure swap-rate law and
\eqref{eq:nr-call-identity} is the swaption quoting identity.
\Cref{cor:nr-sufficient,thm:nr-joint} say:

\begin{enumerate}
\item \emph{Prices and normalized call-price curves require no
inverse-annuity moments.}
Assume the compact model--contract setting and joint continuity of
\cref{thm:nr-joint}, together with uniform integrability, over the
model--contract class, of the annuities and legs, positive
normalization, and its boundary-free primitive range.  These
hypotheses supply the $s=1$ uniform ratio-tail condition in
\eqref{eq:nr-joint-tail}.  Indeed, for $R\geq M/\delta$,
\[
 \E_Q\bigl[|B_a|\one\{|B_a|>RA_a\}\bigr]
 \leq
 2\,\E_Q\bigl[|B_a|\one\{|B_a|>M\}\bigr]
 +M\,Q(A_a\leq\delta),
\]
by splitting on $\{A_a>\delta\}$ and $\{A_a\leq\delta\}$; the first
term is uniformly small in $M$ by the assumed uniform integrability of
the legs, and the second is handled by
\[
 \sup_{(Q,a)\in\cK\times\mathsf A}Q(A_a\leq\delta)
 \xrightarrow[\ \delta\downarrow0\ ]{}0,
\]
which holds because $(Q,a)\mapsto Q(A_a\leq\delta)$ is upper
semicontinuous (portmanteau on the closed set $\{A_a\leq\delta\}$,
after the continuous-mapping step of \cref{thm:nr-joint}) and
decreases pointwise to $Q(A_a=0)=0$, so that for each $\eta>0$ the
nested closed sets
$\{(Q,a):Q(A_a\leq\delta)\geq\eta\}$ have empty intersection and one
of them is therefore already empty by compactness of
$\cK\times\mathsf A$.  \Cref{thm:nr-joint} then gives joint
$W_1$-convergence of the annuity-measure marginals and uniform
convergence of their normalized call-price curves.  As the
normalizations converge, the associated physically settled swaption
prices converge at each fixed strike, uniformly over bounded strike
sets.  No additional moment assumption on $1/A_a$ is imposed.  Related
finite-model calculations, such as the discounting-switch analysis
of \citet{Piterbarg2020}, use an annuity-ratio measure-change density
and approximate its conditional projection onto the swap rate, in
the tradition of \citet[Chapter~16]{AndersenPiterbarg2010}.  The
present result is an asymptotic stability statement, not a valuation
formula for a finite switch or a bound on that approximation error.
\item \emph{Quadratic CMS-type statistics are boundary-priced.}
For the raw second moment,
\[
 \int z^2\dd\Gamma_a
 =\frac{\E_Q[B_a^2/A_a]}{\E_Q[A_a]};
\]
in the boundary-free regime its stability is equivalent to uniform
integrability of $\{B_a^2/A_a\}$.  At the level of the abstract
primitive-law model, \cref{ex:nr-two-atom} shows
convergence of the entire normalized call-price curve while this
quadratic statistic remains asymptotically displaced by one.  Thus a
second-order CMS approximation
depending nontrivially on that statistic need not converge
\citep{Hagan2003}.  A particular CMS correction must be assessed by
its payoff growth; the theorem does not identify every correction
with the second moment.  This is the same higher-moment
integrability obstruction highlighted, in a different setting, by
\citet{AndersenPiterbarg2007}, and it connects with the
moment-extraction discussion of
\citet[fn.~16]{TrolleSchwartz2014}.
\item \emph{Algebraic cancellation and the settlement distinction.}
For a cash-settled payoff $G(S)(S-K)^+$ with $G(S)>0$ and
$0<\E[G(S)]<\infty$, setting $A=G(S)$ and $B=S G(S)$ gives the
perspective identity
$A(B/A-K)^+=G(S)(S-K)^+$.  Thus \cref{prop:nr-cancellation} applies
to the normalized $G$-weighted law; tradability is not one of its
mathematical hypotheses.  The distinction is instead that $G(S)$ is
not a traded numéraire, so this normalized weighting is not
automatically a market pricing measure.  In the reverse problem of
recovering a forward density from a quoted cash-settled price surface,
inverse-annuity integral conditions appear in the no-arbitrage
analysis of \citet{Mercurio2008}.  Under its stated weak-convergence
and integrability hypotheses, the present theorem gives forward
stability of the specified weighted laws; it neither constructs a
market pricing measure nor removes or characterizes those reverse
conditions.
\end{enumerate}

\section*{AI-use disclosure}
\addcontentsline{toc}{section}{AI-use disclosure}

The author used Anthropic Claude Code and OpenAI Codex as
interactive research and writing tools.  They assisted with
exploratory discussion, testing and refinement of ideas, literature
and source organization, code development and verification,
mathematical error checking, and editorial revision.

\bibliographystyle{plainnat}
\phantomsection
\addcontentsline{toc}{section}{References}
\bibliography{references}

\end{document}